\documentclass[12pt]{amsart}

\usepackage{amsmath,amssymb,amsbsy,amsfonts,amsthm,latexsym,
         amsopn,amstext,amsxtra,euscript,amscd,amsthm, mathtools, dsfont, fullpage}
\usepackage[utf8]{inputenc}
\usepackage{tikz}
\usepackage{pgf}
\usepackage{url}
\date{}

\allowdisplaybreaks

\numberwithin{equation}{section}

\newtheorem{lem}{Lemma}
\newtheorem{lemma}[lem]{Lemma}

\newtheorem{prop}{Proposition}

\newtheorem{thm}{Theorem}
\newtheorem{theorem}[thm]{Theorem}

\newtheorem*{example*}{Example}

\title[A discrete approach to broken stick problems]{MacMahon Partition Analysis: a discrete approach to broken stick problems}
\author{William Verreault}

\address{
William Verreault \\
D\'epartement de math\'ematiques et de statistique \\ 
Universit\'e Laval   \\ 
Qu\'ebec G1V 0A6, Canada}
\email{william.verreault.2@ulaval.ca}

\begin{document}

\begin{abstract}
     We propose a discrete approach to solve problems on forming polygons from broken sticks, which is akin to counting polygons with sides of integer length subject to certain Diophantine inequalities. Namely, we use MacMahon's Partition Analysis to obtain a generating function for the size of the set of segments of a broken stick subject to these inequalities. 
In particular, we use this approach to show that for $n\geq k\geq 3$, the probability
that a $k$-gon cannot be formed from a stick broken into $n$ parts is given by $n!$ over a product of linear combinations of partial sums of generalized Fibonacci numbers, a problem which proved to be very hard to generalize in the past.
\end{abstract}

\maketitle
\section{Introduction}
Partition Analysis is a computational method to solve problems that involve linear homogeneous Diophantine inequalities and equations. Introduced by MacMahon in his famous book \textit{Combinatory Analysis} \cite{macmahon2001combinatory}, and later used by Stanley in his proof of the Anand–Dumir–Gupta conjecture \cite{stanley_linear_1973}, the method died down for decades until Andrews et al. devoted a series of papers to applications of Partition Analysis. They started with a new proof of the Lecture Hall Partition Theorem of Bousquet-Mélou and Eriksson \cite{bousquet-melou_lecture_1997} and even provided a computer algebra implementation of the method (see \cite{sagan_macmahons_1998} for a great exposition to the history and ideas of MacMahon's Partition Analysis). 

On the other hand, in broken stick problems, one breaks a stick at random at $n-1$ points to obtain $n$ pieces and then asks diverse geometric probability questions involving this broken stick. Henceforth, we will focus on forming polygons from these pieces. The basis of these questions is the famous problem from the Mathematical Tripos on finding the probability that a triangle can be formed from a stick broken at random in three segments, which was quickly generalized to forming an $n$-gon from the $n$ pieces of a broken stick. The probability in the latter case is simply $1-n/2^{n-1}$, and the most popular solution to this problem comes from the note \textit{The Broken Spaghetti Noodle} of D'Andrea and Gómez \cite{dandrea_broken_2006}. For more background and motivation on these kinds of problems, an interested reader may wish to read the introduction of \cite{Verreault}.

\subsection{A generalization}
    It is natural to generalize these problems by using $3\leq k\leq n$ pieces to form a $k$-gon out of the $n$ pieces, but they turn out to be much harder than expected and only a few have found an answer. In project reports by the Illinois Geometry Lab \cite{kong2013random, page_calculus_2015}, two extreme cases were considered, namely the probability that given a stick broken up into $n$ pieces, there exist three segments that can form a triangle (we call it the ``there exist'' problem), as well as the probability that all choices of three segments can form a triangle (the ``for all'' problem). It was stated without proof that the answers are respectively \begin{equation} \label{Illinois}
    1-n!\prod_{j=2}^n (F_{j+2}-1)^{-1} \text{ and } \binom{2n-2}{n}^{-1},\end{equation} where $F_j$ is the $j$th Fibonacci number. 
    Crowdmath's $2017$ project\footnote{Crowdmath is an open project for high school and college students to collaborate on large research projects and is based on the model of Polymath. $2017$'s project was called \textit{The Broken Stick Problem}.} answered many broken stick type questions, but one of the only problems they could not solve had to do with forming polygons from a broken stick. 
   However, the general case of the ``for all'' problem was solved recently by using results on random divisions of an interval and order statistics.
       \begin{thm}{}{\cite{Verreault}}
    Let $n\geq k\geq 3$ be positive integers. The probability that every choice of $k$ segments from a stick broken up into $n$ parts will form a $k$-gon is given by
$$
\frac{n(n-1)\cdots(n-k+3)}{(n-k+2)}\sum_{j=1}^{n-k+2}\frac{(-1)^{j+1}}{j^{k-3}}\binom{n-k+2}{j}\left(\frac{n-k+2}{j}+1\right)_{k-2}^{-1},
$$
where  $$\left(\frac{n-k+2}{j}+1\right)_{k-2}=\prod_{i=0}^{k-3}\left(\frac{n-k+2}{j}+1+i\right)$$ stands for the rising factorial function or Pochhammer symbol.
    \end{thm}
It was pointed out that the method could be adapted to answer similar questions such as the ``there exist'' problem, but in practice it was much harder to apply.

In this paper, we show that MacMahon's Partition Analysis is well suited to solve broken stick problems by considering a discrete setting and exhibiting a connection with systems of linear Diophantine inequalities. This discrete setting is more generally related to counting polygons with sides of integer length, as studied in \cite{andrews_macmahons_2001} for example. In particular, we apply it to solve the ``there exist'' problem in full generality. As a purely aesthetic choice, we will consider the complimentary probability that no choices of $k$ segments from a stick broken up into $n$ parts will form a $k$-gon.

\subsection{Preliminaries and the ``there exist'' problem}
We define the generalized Fibonacci numbers $F_n^{(k)}$ as it is typically done:
$$
F_n^{(k)}:=\left.
\begin{cases}
    0     &0 \leq n\leq k-2,\\
    1     &n=k-1,\\
    \sum_{i=1}^kF_{n-i}^{(k)}   &n\geq k.
\end{cases}
\right.
$$
Next, let 
$$
f_{k}(i):=\left.
\begin{cases}
    0     &0 \leq i\leq k-2,\\
    \sum_{n=k-1}^i F_n^{(k)}     &i\geq k-1
\end{cases}
\right.
$$
be restricted partial sums of the $F_n^{(k)}$,
and let $$
g_k(j):=1+\sum_{\ell=2}^j f_{k}(n-\ell) \qquad (2\leq j\leq n-k+1)
$$
and
$$
h_{k}(\ell):=f_{k}(n)+\sum_{j=2}^\ell g_k(k+1-j) \qquad (2\leq \ell \leq k-2).
$$
The preceding definitions allow us to state our main result in terms of linear combinations of generalized Fibonacci numbers.

\begin{theorem} \label{Main}
Let $n\geq k\geq 3$ be positive integers. The probability that no choices of $k$ segments from a stick broken up into $n$ parts will form a $k$-gon is given by 
$$
\frac{n!}{f_{k-1}\left(k-2\right)\cdots f_{k-1}\left(n\right)h_{k-1}(2)\cdots h_{k-1}(k-2)},
$$
where the $h_{k-1}(\ell)$ are to be removed for $k=3$.
\end{theorem}
Unfortunately, there is not a (nice enough) closed form for partial sums of generalized Fibonacci numbers to express the previous expression in a simpler form, but one is well known for typical Fibonacci numbers $F_n^{(2)}$:
$$
f_{2}(i)=\sum_{n=1}^i F_n^{(2)} = F_{i+2}^{(2)}-1
$$
for $i\geq 2$, and so the expression obtained by the Illinois Geometry lab \eqref{Illinois} follows.
Also note that while the probability might seem hard to calculate, it is simple from a computational point of view, since generalized Fibonacci numbers are easily calculated and the sums $f_{k}(i)$ are documented on the OEIS.

\begin{example*}
If we want to make a $4$-gon, we have
$$(f_3(i))_{i\geq 2}=(1,2,4,8,15,28,52,96,177,\ldots) \qquad (\text{OEIS } A008937 \text{ \cite{oeis}}),$$ and so, if $P(4,n)$ stands for the probability that no choices of $4$ segments from a stick broken up into $n$ parts will form a $4$-gon, then
 $$P(4,n)=\frac{n!}{1\cdot2\cdot 4\cdots f_{3}(n)(1+f_{3}(n-2)+f_{3}(n))}.$$ 
Thus, we can easily calculate
$$P(4,4)=\frac{4!}{1\cdot 2\cdot 4\cdot (1+1+4)}=\frac{1}{2}, \;
P(4,5)=\frac{15}{88}, \;
P(4,6)=\frac{3}{80},$$
and so on.
\end{example*}

Our main tool to prove the previous theorem will be the following lemma on partitions subject to a system of Diophantine inequalities, which we will establish using Partition Analysis.

\begin{lemma} \label{lem}
The number of partitions of $N$ into $n$ positive parts $a_i$ and under the additional linear constraints 
\begin{align*}
    a_1&\geq a_2+a_3+\cdots+a_k, \\
    a_2&\geq a_3+a_4+\cdots+a_{k+1}, \\
    \vdots& \\
    a_{n-k+1}&\geq a_{n-k+2}+a_{n-k+3}+\cdots+a_n \\
    a_{n-k+2}&\geq a_{n-k+3}\\
    a_{n-k+3}&\geq a_{n-k+4}\\
    \vdots&\\
    a_{n-1}&\geq a_n
\end{align*}
 equals the number of partitions of $N$ into parts taken from the set of integers $$\{f_{k-1}\left(k-2\right),f_{k-1}\left(k-1\right),\ldots,f_{k-1}\left(n\right),h_{k-1}(2),h_{k-1}(3),\ldots,h_{k-1}(k-2)\}.$$
\end{lemma}

We end this section with a brief overview of the mathematics behind MacMahon's operator that we will use in Partition Analysis.
The linear operator $\Omega_{\geq}$ is defined as acting on multiple Laurent series as follows:
    $$
    \underset{\geq}{\Omega}\sum_{n_1,\ldots,n_r=-\infty}^{\infty}A_{n_1,\ldots,n_r}\lambda_1^{n_1}\cdots\lambda_r^{n_r} = \sum_{n_1,\ldots,n_r=0}^{\infty}A_{n_1,\ldots,n_r}.
    $$   
    The $A_{n_1,\ldots,n_r}$ can be thought of as rational functions of several variables independent of the $\lambda_i$ and the series are to be treated analytically (because the method relies on unique Laurent series expansion of rational functions).
    Also note that we could specify a domain in each case to guarantee absolute convergence of the sums in an open neighborhood of the circles $|\lambda_i|=1$, so that the $\Omega_{\geq}$ operator is well-defined.
    
While most of the time $\Omega_{\geq}$ acts on $\lambda_i$, we will also use $\mu_i$ to indicate that these variables encode a different kind of linear constraint than the $\lambda_i$.

\section{Discrete setting and sketch of argument}
Henceforth, we consider a stick of integer length $N$ broken at integer increments into $n$ pieces $a_1,\ldots,a_n$.
We may consider broken stick problems on forming polygons in this discrete setting. We call ``Set of all outcomes'' the set of all possible $n$-tuples of pieces that can form a stick of length $N$, namely the set of compositions of $N$ into $n$ parts $$C(N,n):=    \left\{\left(a_1,\ldots,a_n\right)\in \mathbb{Z}^n:a_i\geq 1, \sum a_i=N\right\}.
$$
To every such problem, we can associate a ``Set of positive outcomes'' which represents the subsets of $C(N,n)$ that satisfy given Diophantine inequalities. Then, the probabilities we seek will take the form \begin{equation} \label{prob} 
\frac{\#\text{Set of positive outcomes}}{\#\text{Set of all outcomes}}=\frac{ \#\text{Set of positive outcomes}}{\binom{N-1}{n-1}}.
\end{equation}
Below, we shall see concrete examples of ``Set of positive outcomes'' associated to certain problems. 
To go back to the continuous setting, we will look at the behavior of \eqref{prob} as $N\to\infty$, hence we only need asymptotics on $\#\text{Set of positive outcomes}$, which we obtain via a generating function for this set.

It should be noted that very recently, the authors in \cite{petersen_broken_2020} have considered a discrete variant of the \textit{Broken Spaghetti Noodle} problem to give a purely combinatorial proof. However, the proof we give below as an example of this discrete setting being applied is very different in nature.

\subsection{Example} We remind the reader that the probability that an $n$-gon can be formed from a stick broken into $n$ pieces is
\begin{align} \label{spaghetti}
    1-\frac{n}{2^{n-1}}.
\end{align}
To form an $n$-gon, it suffices to check that $a_i\leq n-a_i$ for $i=1,\ldots,n$ by the generalized triangle inequality, so the ``Set of positive outcomes'' is 
$$
    \left\{(a_1,\ldots,a_n)\in C(N,n):a_i\leq n-a_i \text{ for } i=1,\ldots,n\right\}.
$$
While one could obtain the generating function for the size of this set with Partition Analysis, we simply mention for the sake of this example that it has been studied by Andrews, Paule and Riese in \cite{andrews2001macmahon} as a generalization of Hermite's problem. 
They showed that this generating function is 
    $$
    \frac{q^{n}}{(1-q)^{n}}-n\frac{q^{2n-1}}{(1-q)^n(1+q)^{n-1}}.
    $$

The first term is the generating function for compositions, so comparing with \eqref{spaghetti}, we only have to show that the coefficient of $q^n$ in $$\frac{q^{2n-1}}{(1-q)^n(1+q)^{n-1}}\binom{N-1}{n-1}^{-1}$$ is $2^{1-n}$ as $N\to\infty$. This is easily seen to hold after doing partial fraction decomposition, clearing denominators, using binomial expansions and setting $q=1$.

We now give another example of the discrete setting being applied, but this time we explicitly use Partition Analysis.

\subsection{Example} \label{Example k=4}
We will treat the case $k=4$ of Theorem \ref{Main}.
It can be hard to verify that no choices of $4$ pieces can form a $4$-gon, since we need to check that they fail the triangle inequality every time. But if the segments of the broken stick are ordered, then it suffices to verify that, for consecutive groups of $4$ segments,
the biggest piece is smaller than or equal to the sum of the $3$ other parts of the group (which also excludes degenerate triangles). Thus, for simplicity, we assume that $a_1\geq a_2\geq\ldots\geq a_n\geq 1$. 
      
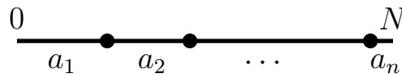
\begin{center}
\begin{tikzpicture} [ultra thick]
  \draw (0,0)--(5,0) ;
  \filldraw (1.2,0) circle (2pt) ;
  \draw (0.6,-0.3) node {$a_1$} ;
  \draw (1.8,-0.3) node {$a_{2}$} ;
    \draw (4.9,-0.3) node {$a_n$} ; 
    \filldraw (4.7,0) circle (2pt) ;
      \filldraw (2.3,0) circle (2pt) ;
      \draw (3.3,-0.3) node {$\cdots$} ;
  \draw (3.5,0.3) node {} ;
  \draw (5,0.3) node {$N$} ;
    \draw (0,0.3) node {$0$} ;
\end{tikzpicture}
\end{center}

The preceding discussion tells us that the ordered ``Set of positive outcomes'' is
$$
    S_4(N):=\left\{(a_1,\ldots,a_n)\in C(N,n):a_1\geq\ldots\geq a_n, a_{i}\geq \sum_{j=1}^3a_{i+j} \text{ for } i=1,\ldots,n-3\right\}.
$$
To deal with this assumption on the ordering, one needs to observe that
$$
\#\text{Set of positive outcomes}=4!\#S_4(N).
$$
We consider the associated generating function
$$
G_4\left(q\right):=\sum_{N\geq 4}\#S_4(N)q^N
=\sum_{\substack{a_1\geq a_2\geq \ldots \geq a_n\geq 1 \\ a_1\geq a_2+a_3+a_4 \\ \vdots \\ a_{n-3}\geq a_{n-2}+a_{n-1}+a_n}}q^{a_1+\cdots+a_n}
$$
and we use Partition Analysis to find a closed form representation. 

MacMahon derived many identities for his Partition Analysis. The ones we will need for this example are
\begin{equation} \label{Identity_2}
 \underset{\geq}{\Omega} \frac{1}{\left(1-\lambda x\right)\left(1-\frac{y}{\lambda}\right)}=\frac{1}{\left(1-x\right)\left(1-xy\right)}
\end{equation}
and
\begin{equation} \label{Identity_4}
 \underset{\geq}{\Omega} \frac{1}{\left(1-\lambda x\right)\left(1-\frac{y}{\lambda}\right)\left(1-\frac{z}{\lambda}\right)\left(1-\frac{w}{\lambda}\right)}=\frac{1}{\left(1-x\right)\left(1-xy\right)\left(1-xz\right)\left(1-xw\right)}.
\end{equation}
These identities follow from geometric series summation, and it is easy to see how to generalize them. One can also check \cite{andrews2001macmahon} for many more such identities.

We encode the Diophantine inequalities of $G_4\left(q\right)$ in new parameters to use MacMahon's operator:
\begin{align*}
G_4\left(q\right) 
=\underset{\geq}{\Omega}\sum_{a_i\geq 0}q^{a_1+\cdots+a_n}\lambda_1^{a_1-a_2-a_3-a_4}\lambda_2^{a_2-a_3-a_4-a_5}\cdots\lambda_{n-3}^{a_{n-3}-a_{n-2}-a_{n-1}-a_{n}}\mu_{n-2}^{a_{n-2}-a_{n-1}}\mu_{n-1}^{a_{n-1}-a_{n}}.
\end{align*}
Notice that if $a_{i}\geq \sum_{j=1}^3a_{i+j}$ for $i=1,\ldots,n-3$, then $a_1\geq a_2\geq\ldots\geq a_{n-2}$ is satisfied, which explains the presence of the variables $\mu_{n-2}$ and $\mu_{n-1}$ to encode the remaining inequalities $a_{n-2}\geq a_{n-1}\geq a_n$. 

Using geometric series summation, the generating function is equal to
\begin{align*}
\underset{\geq}{\Omega}&\frac{1}{
\left(1-q\lambda_1\right)\left(1-q\frac{\lambda_2}{\lambda_1}\right)\left(1-q\frac{\lambda_3}{\lambda_2\lambda_1}\right)\left(1-q\frac{\lambda_4}{\lambda_3\lambda_2\lambda_1}\right)\left(1-q\frac{\lambda_5}{\lambda_4\lambda_3\lambda_2}\right)\cdots\left(1-q\frac{\lambda_{n-3}}{\lambda_{n-4}\lambda_{n-5}\lambda_{n-6}}\right)} \\
\times &\frac{1}{\left(1-q\frac{\mu_{n-2}}{\lambda_{n-3}\lambda_{n-4}\lambda_{n-5}}\right)\left(1-q\frac{\mu_{n-1}}{\mu_{n-2}\lambda_{n-3}\lambda_{n-4}}\right)\left(1-q\frac{1}{\mu_{n-1}\lambda_{n-3}}\right)}.
\end{align*}
Next, applying \eqref{Identity_4} iteratively to get rid of $\lambda_1,\lambda_2,\ldots,\lambda_{n-3}$, we get
\begin{align*}
&\frac{1}{\left(1-q^1\right)}\underset{\geq}{\Omega}\frac{1}{
\left(1-q^{1+1}\lambda_2\right)\left(1-q^{1+1}\frac{\lambda_3}{\lambda_2}\right)\left(1-q^{1+1}\frac{\lambda_4}{\lambda_3\lambda_2}\right)\left(1-q\frac{\lambda_5}{\lambda_4\lambda_3\lambda_2}\right)\cdots\left(1-q\frac{\lambda_{n-3}}{\lambda_{n-4}\lambda_{n-5}\lambda_{n-6}}\right) }\\
\times &\frac{1}{(1-q\frac{\mu_{n-2}}{\lambda_{n-3}\lambda_{n-4}\lambda_{n-5}})(1-q\frac{\mu_{n-1}}{\mu_{n-2}\lambda_{n-3}\lambda_{n-4}})(1-q\frac{1}{\mu_{n-1}\lambda_{n-3}})}\\
=&\frac{1}{\left(1-q^1\right)\left(1-q^{1+1}\right)}\\
\times &\underset{\geq}{\Omega}\frac{1}{
\left(1-q^{1+1+2}\lambda_3\right)\left(1-q^{1+1+2}\frac{\lambda_4}{\lambda_3}\right)\left(1-q^{1+2}\frac{\lambda_5}{\lambda_4\lambda_3}\right)\left(1-q\frac{\lambda_6}{\lambda_5\lambda_4\lambda_3}\right)\cdots\left(1-q\frac{\lambda_{n-3}}{\lambda_{n-4}\lambda_{n-5}\lambda_{n-6}}\right)}\\
\times &\frac{1}{(1-q\frac{\mu_{n-2}}{\lambda_{n-3}\lambda_{n-4}\lambda_{n-5}})(1-q\frac{\mu_{n-1}}{\mu_{n-2}\lambda_{n-3}\lambda_{n-4}})(1-q\frac{1}{\mu_{n-1}\lambda_{n-3}})}
\\
=&\frac{1}{\left(1-q^1\right)\left(1-q^{1+1}\right)\left(1-q^{1+1+2}\right)}\\
\times &\underset{\geq}{\Omega}\frac{1}{
\left(1-q^{1+1+2+4}\lambda_4\right)\left(1-q^{1+2+4}\frac{\lambda_5}{\lambda_4}\right)\left(1-q^{1+4}\frac{\lambda_6}{\lambda_5\lambda_4}\right)\left(1-q\frac{\lambda_7}{\lambda_6\lambda_5\lambda_4}\right)\cdots\left(1-q\frac{\lambda_{n-3}}{\lambda_{n-4}\lambda_{n-5}\lambda_{n-6}}\right)}\\
\times &\frac{1}{(1-q\frac{\mu_{n-2}}{\lambda_{n-3}\lambda_{n-4}\lambda_{n-5}})(1-q\frac{\mu_{n-1}}{\mu_{n-2}\lambda_{n-3}\lambda_{n-4}})(1-q\frac{1}{\mu_{n-1}\lambda_{n-3}})},
\end{align*}
and so on.
We have left the exponents under this form to show the connection with Generalized Fibonacci numbers. Hence, after $n-3$ applications of \eqref{Identity_4}, we obtain
\begin{align*}
     G_4(q)=&\frac{1}{
 \left(1-q^{f_3(2)}\right)\left(1-q^{f_3(3)}\right)\left(1-q^{f_3(4)}\right)\left(1-q^{f_3(5)}\right)\cdots\left(1-q^{f_3(n-2)}\right)}\\
 \times &\underset{\geq}{\Omega}\frac{1}{\left(1-q^{f_{3}(n-1)}\mu_{n-2}\right)\left(1-q^{1+f_{3}(n-3)+f_3(n-2)}\frac{\mu_{n-1}}{\mu_{n-2}}\right)\left(1-q^{1+f_3(n-2)}\frac{1}{\mu_{n-1}}\right)},
\end{align*}
where we have implicitly used the fact that $$f_3(j)=1+f_3(j-1)+f_3(j-2)+f_3(j-3)$$ for all $j\geq 3$ (see Lemma \ref{lem 3}), and that $(f_3(i))_{i\geq 2}=(1,2,4,8,15,28,\ldots)$.

Finally, we only need to apply \eqref{Identity_2} twice 
to get the expression
 $$
 G_4(q)=\frac{1}{
  \left(1-q^{f_3(2)}\right)\left(1-q^{f_3(3)}\right)\left(1-q^{f_3(4)}\right)\cdots\left(1-q^{f_3(n)}\right)\left(1-q^{1+f_{3}(n-2)+f_3(n)}\right)}.
 $$

Now that we have obtained this closed form representation for $G_4(N)$, which showcased how Partition Analysis might be applied to solve broken stick problems, we move on to the more general case. This will also show how to finish the proof of Theorem \ref{Main} after obtaining a suitable representation for the generating function of the size of the ``Set of positive outcomes''.

\section{Proof of Theorem \ref{Main}}
Most of the results we present in this section can be proved using double induction on $n\geq k\geq 3$ if desired, but we find it to be unenlightening and unnecessarily messy. We prefer to break down the steps in a natural way and directly apply identities akin to \eqref{Identity_4}.

\subsection{Proof of Lemma \ref{lem}}
Let $p_k(N,n)$ denote the number of partitions of $N$ into $n$ positive parts that respect the system of inequalities in Lemma \ref{lem}. It is easy to see that the following generating function keeps track of $p_k(N,n)$:
$$
\sum_{N\geq k}p_k(N,n)q^N
=\sum_{\substack{a_1\geq a_2\geq \ldots \geq a_n\geq 1 \\ a_1\geq a_2+\cdots+a_k \\ \vdots \\ a_{n-k+1}\geq a_{n-k+2}+\cdots+a_n}}q^{a_1+\cdots+a_n},
$$
since the inequalities $a_1\geq a_2\geq \ldots \geq a_{n-k+2}$ are vacuously satisfied.
Following the previous example, we add new parameters to encode the system of inequalities and use Partition Analysis. We get
\begin{align*} \begin{split}
    &\underset{\geq}{\Omega}\sum q^{a_1+\cdots+a_n}\lambda_1^{a_1-a_2-\cdots-a_k}\lambda_2^{a_2-a_3-\cdots-a_{k+1}}\cdots \lambda_{n-k+1}^{a_{n-k+1}-a_{n-k+2}-\cdots-a_{n}} \\
    &\times \mu_{n-k+2}^{a_{n-k+2}-a_{n-k+3}}\mu_{n-k+3}^{a_{n-k+3}-a_{n-k+4}}\cdots\mu_{n-1}^{a_{n-1}-a_{n}} 
\end{split} \\ \begin{split}
&= \underset{\geq}{\Omega}\frac{1}{
\left(1-q\lambda_1\right)\left(1-q\frac{\lambda_2}{\lambda_1}\right)\cdots\left(1-q\frac{\lambda_k}{\lambda_{k-1}\cdots\lambda_1}\right)\left(1-q\frac{\lambda_{k+1}}{\lambda_k\cdots\lambda_2}\right)\cdots\left(1-q\frac{\lambda_{n-k+1}}{\lambda_{n-k}\cdots\lambda_{n-2k+2}}\right) } \\
&\times \frac{1}{\left(1-q\frac{\mu_{n-k+2}}{\lambda_{n-k+1}\cdots\lambda_{n-2k+3}}\right) \left(1-q\frac{\mu_{n-k+3}}{\mu_{n-k+2}\lambda_{n-k+1}\cdots\lambda_{n-2k+4}}\right) 
\cdots \left(1-q\frac{\mu_{n-1}}{\mu_{n-2}\lambda_{n-k+1}\lambda_{n-k}}\right)\left(1-q\frac{1}{\mu_{n-1}\lambda_{n-k+1}}\right)
}. \end{split}
\end{align*}

We would like to apply a reduction identity like \eqref{Identity_4} to get rid of $\lambda_1,\ldots,\lambda_{n-k+1}$. We use a generalization of that identity which easily follows from induction and geometric series summation. 
\begin{lem} \label{lem 2}
For every $k\geq 2$,
\begin{equation} \label{Identity_k}
 \underset{\geq}{\Omega} \frac{1}{\left(1-\lambda x_1\right)\left(1-\frac{x_2}{\lambda}\right)\cdots\left(1-\frac{x_k}{\lambda}\right)}=\frac{1}{\left(1-x_1\right)\left(1-x_1x_2\right)\cdots\left(1-x_1\cdots x_k\right)}.
\end{equation}
\end{lem}
We also have the following result on partial sums of the $f_{k-1}(i)$ which is to be used jointly with the previous lemma.
\begin{lem} \label{lem 3}
Let $k\geq 2$. Then for every $j\geq k-1$, we have
$$
f_k(j)=1+\sum_{i=1}^{k}f_k(j-i).
$$
\end{lem}
This result is readily seen to follow from the definitions of $f_k(i)$ and $F_n^{(k)}$ after a careful rereading. It can also be proven by induction as follows, but it is still no more than a careful unpacking of definitions.
\begin{proof}
There is nothing to prove for $j=k-1$ since both sides of the expression reduce to $1$. On the other hand, since
$$
f_k(j+1)=f_k(j)+F_{j+1}^{(k)}
$$
and $$F_{j+1}^{(k)}=\sum_{i=1}^{k}F_{j+1-i}^{(k)},$$ we find that the result holds for $j+1$ if we assume it does for some $j\geq k-1$.
\end{proof}
The previous lemmas allow us to obtain the following proposition.
\begin{prop} \label{prop 1}
For every $n\geq k\geq 3$, we have
\begin{align*}
    \begin{split}
&\underset{\geq}{\Omega}\frac{1}{
\left(1-q\lambda_1\right)\left(1-q\frac{\lambda_2}{\lambda_1}\right)\cdots\left(1-q\frac{\lambda_k}{\lambda_{k-1}\cdots\lambda_1}\right)\left(1-q\frac{\lambda_{k+1}}{\lambda_k\cdots\lambda_2}\right)\cdots\left(1-q\frac{\lambda_{n-k+1}}{\lambda_{n-k}\cdots\lambda_{n-2k+2}}\right)} \\
&\times \frac{1}{\left(1-q\frac{\mu_{n-k+2}}{\lambda_{n-k+1}\cdots\lambda_{n-2k+3}}\right) \left(1-q\frac{\mu_{n-k+3}}{\mu_{n-k+2}\lambda_{n-k+1}\cdots\lambda_{n-2k+4}}\right) 
\cdots \left(1-q\frac{\mu_{n-1}}{\mu_{n-2}\lambda_{n-k+1}\lambda_{n-k}}\right)\left(1-q\frac{1}{\mu_{n-1}\lambda_{n-k+1}}\right)
} \end{split} \\
&=\frac{1}{\left(1-q^{f_{k-1}(k-2)}\right)\left(1-q^{f_{k-1}(k-1)}\right)\cdots\left(1-q^{f_{k-1}(n-1)}\right)} \underset{\geq}{\Omega}\frac{1}{\left(1-q^{f_{k-1}(n)}\mu_{n-k+3}\right)} \\
&\times \frac{1}{\left(1-q^{g_{k-1}(k-2)}\frac{\mu_{n-k+4}}{\mu_{n-k+3}}\right)
\left(1-q^{g_{k-1}(k-3)}\frac{\mu_{n-k+5}}{\mu_{n-k+4}}\right)
\cdots
\left(1-q^{g_{k-1}(3)}\frac{\mu_{n-1}}{\mu_{n-2}}\right)
\left(1-q^{g_{k-1}(2)}\frac{1}{\mu_{n-1}}\right)}.
\end{align*}
\end{prop}

\begin{proof}
We can sort the terms in the denominator of the given expression depending on their general form. We have
   \begin{align} \label{form 1} 1-q\frac{\lambda_{\ell}}{\lambda_{\ell-1}\cdots\lambda_{1}} \qquad (2\leq \ell \leq k), \end{align}
   \begin{align} \label{form 2} 1-q\frac{\lambda_{\ell}}{\lambda_{\ell-1}\cdots\lambda_{\ell-k+1}} \qquad (k+1\leq \ell \leq n-k+1),\end{align}
   or
    \begin{align} \label{form 3} 1-q\frac{\mu_{\ell}}{\mu_{\ell-1}\lambda_{n-k+1}\cdots\lambda_{\ell-k+1}} \qquad (n-k+3\leq \ell \leq n-1).\end{align}
This obviously excludes $1-q\lambda_1$,
\begin{align} \label{form separate}
1-q\frac{\mu_{n-k+2}}{\lambda_{n-k+1}\cdots\lambda_{n-2k+3}},
\end{align}
and
\begin{align} \label{form last}
    1-q\frac{1}{\mu_{n-1}\lambda_{n-k+1}},
\end{align}
but the first term is straightforward to treat separately and \eqref{form separate} is nothing but \eqref{form 2} in disguise (once we extend $\ell$ to $n-k+2$). In the same way, we might as well treat \eqref{form last} as being of the form \eqref{form 3} since $\underset{\geq}{\Omega}$ is not acting on $\mu_\ell$ for $\ell>n-k+2$ in this Proposition.

We start with \eqref{form 1}. Writing the exponent of $q$ in $1-q\lambda_1$ as $f_{k-1}(k-2)$, it is easy to see iteratively, using Lemma \ref{lem 2}, that we can associate a forward shift by $k-3$ between the variable given to $f_{k-1}(\cdot)$ and the index of the parameter in the numerator of each term of the general form \eqref{form 1}. Since there are $\ell-1$ parameters $\lambda_i$ in each denominator, we see using the same iterative reasoning with $\ell$ applications of Lemma \ref{lem 2} (acting on $\lambda_1,\ldots,\lambda_\ell$ in that order) that the exponent of $q$ will be given by the sum of $1$ with the $\ell-1$ preceding exponents. Since we have defined $f_{k-1}(i)=0$ when $i\leq k-3$, we might as well say that it is the sum of $1$ and the $k-1$ preceding ones where we substitute $k-\ell$ times a $0$ for the terms that don't initially appear; that is, the exponents will take the form
$$
1+f_{k-1}(\ell+k-4)+f_{k-1}(\ell+k-5)+\cdots+f_{k-1}(\ell-2),
$$
which is just $f_{k-1}(\ell+k-3)$ by Lemma \ref{lem 3}. 
The exact same reasoning works for \eqref{form 2} (but this time there is no need to add any $0$).

Finally, we treat \eqref{form 3}. Notice that each denominator in this case is composed of $n-\ell+1$ terms starting with $\lambda_{n-k+1}$ and decreasing from there (we are obviously not including the $\mu_i$ in this discussion since $\underset{\geq}{\Omega}$ will not be acting on them). Hence, we see by a repetition of our preceding argument that the exponent of $q$ will be the sum of $1$ with the $n-\ell+1$ preceding ones, starting at $f_{k-1}(n)$ by our $k-3$ shift acting first on $n-k+1$. This is precisely the definition of $g_{k-1}(n-\ell+2)$. Note that for $\ell=n-k+3$, this gives $g_{k-1}(k-1)$, which is just
$$
1+\sum_{i=2}^{k-1}f_{k-1}(n-i)=f_{k-1}(n)-f_{k-1}(n-1)$$ by Lemma \ref{lem 3}. Hence, to finish the proof, simply apply \eqref{Identity_k} once more, but this time to get rid of $\mu_{n-k+2}$.
\end{proof}
From the previous lemma, it is easy to see why the $h_{k-1}(\cdot)$ terms have to be dropped from Theorem \ref{Main} when $k=3$. In particular, its only two terms involving $\mu_i$ would be of the form \eqref{form separate} and \eqref{form last}, which would clearly not involve any $g_{k-1}(\cdot)$ after applying the $\underset{\geq}{\Omega}$ operator.

The following proposition gives us a final closed form for our generating function.

\begin{prop} \label{Prop 2}
For every $n\geq k\geq 3$, we have
$$
\sum_{N\geq k}p_k(N,n)q^N=
\prod_{i=k-2}^n \frac{1}{1-q^{f_{k-1}(i)}} \prod_{j=2}^{k-2}\frac{1}{1-q^{h_{k-1}(j)}}.
$$
\end{prop}

\begin{proof}
By Proposition \ref{prop 1}, it suffices to show that
\begin{align*}
   & \underset{\geq}{\Omega}\frac{1}{\left(1-q^{f_{k-1}(n)}\mu_{n-k+3}\right)\left(1-q^{g_{k-1}(k-2)}\frac{\mu_{n-k+4}}{\mu_{n-k+3}}\right)
\cdots
\left(1-q^{g_{k-1}(3)}\frac{\mu_{n-1}}{\mu_{n-2}}\right)
\left(1-q^{g_{k-1}(2)}\frac{1}{\mu_{n-1}}\right)} 
\end{align*}
is equal to
$$
\frac{1}{1-q^{f_{k-1}(n)}}\prod_{j=2}^{k-2}\frac{1}{1-q^{h_{k-1}(j)}}. 
$$
To deal with the first term $\left(1-q^{f_{k-1}(n)}\mu_{n-k+3}\right)^{-1}$, simply apply \eqref{Identity_k}. Next, since each $\mu_i$ appears exactly twice, every application of Lemma \ref{lem 2} will only affect the subsequent term of the general form
$$
\frac{1}{1-q^{g_{k-1}(k-\ell)\frac{\mu_{n-k+\ell+2}}{\mu_{n-k+\ell+1}}}}
$$
for some $2\leq \ell \leq k-2$. Thus, we see iteratively that after applying \eqref{Identity_k} $\ell$ times, the exponent of $q$ in such a term will take the form
$$
g_{k-1}(k-\ell)+f_{k-1}(n)+g_{k-1}(k-2)+\cdots+g_{k-1}(k-\ell+1)=h_{k-1}(\ell). \qedhere
$$ 
\end{proof}

To finish the proof of Lemma \ref{lem}, we simply need to observe that the expression given by Proposition \ref{Prop 2} is precisely the generating function for the number of partitions of $N$ into parts taken from the set of integers $$\{f_{k-1}\left(k-2\right),f_{k-1}\left(k-1\right),\ldots,f_{k-1}\left(n\right),h_{k-1}(2),h_{k-1}(3),\ldots,h_{k-1}(k-2)\}.$$

\subsection{Proof of Theorem \ref{Main}}
Obviously, the reasoning made in Example \ref{Example k=4} still applies here. \textit{Mutatis mutandis}, we obtain that the ordered ``Set of positive outcomes'' is 
$$
    S_k(N):=\left\{(a_1,\ldots,a_n)\in C(N,n):a_1\geq\ldots\geq a_n, a_{i}\geq \sum_{j=1}^{k-1}a_{i+j} \text{ for } i=1,\ldots,n-k+1\right\},
$$
while its associated generating function is
$$
G_k\left(q\right):=\sum_{N\geq k}\#S_k(N)q^N
=\sum_{\substack{a_1\geq a_2\geq \ldots \geq a_n\geq 1 \\ a_1\geq a_2+\cdots+a_k \\ \vdots \\ a_{n-k+1}\geq a_{n-k+2}+\cdots+a_n}}q^{a_1+\cdots+a_n}.
$$
It follows that $\#S_k(N)=p_k(N,n)$. But since we assume that the pieces are ordered, we must consider the permutations of the pieces as before. Hence,
$$
\#\text{Set of positive outcomes}=n!\#S_k(N)=n!\cdot 
p_k(N,n).
$$

It is a well known problem and a simple exercise in complex analysis to obtain asymptotics for the number of partitions with restricted summands. For instance, we have the following proposition.
\begin{prop} \cite[p.~258]{flajolet2009analytic}
The number of partitions of $n$ with summands restricted to a finite set of $r$ integers without a common divisor is asymptotically equal to
$$
\frac{1}{P}\frac{n^{r-1}}{(r-1)!},
$$
where $P$ is the product of those $r$ integers.
\end{prop}

Since the set 
$$\{f_{k-1}\left(k-2\right),f_{k-1}\left(k-1\right),\ldots,f_{k-1}\left(n\right),h_{k-1}(2),h_{k-1}(3),\ldots,h_{k-1}(k-2)\}$$
contains $n$ integers without a common divisor (just notice $1$ is in the set), it follows that 
$$
\#S_k(N)\sim \frac{1}{P}\frac{N^{n-1}}{(n-1)!},
$$
where $$P=f_{k-1}\left(k-2\right)f_{k-1}\left(k-1\right)\cdots f_{k-1}\left(n\right)h_{k-1}(2)h_{k-1}(3)\cdots h_{k-1}(k-2).$$
Overall, using that for fixed $s$,
$$
\binom{t}{s}\sim\frac{t^s}{s!},
$$
we get 
$$
\frac{\#\text{Set of positive outcomes}}{\#\text{Set of all outcomes}} \sim \frac{n!}{P}\frac{N^{n-1}}{\left(n-1\right)!}\binom{N-1}{n-1}^{-1}
\sim \frac{n!}{P} \left(\frac{N}{N-1}\right)^{n-1},
$$
which goes to the claimed expression as $N\to\infty$.

\section*{Acknowledgments}
The author would like to thank G.E. Andrews for kindly answering many questions about MacMahon's Partition Analysis.

\bibliographystyle{abbrv}
\bibliography{references.bib}

\begin{thebibliography}{10}

\bibitem{sagan_macmahons_1998}
G.~E. Andrews.
\newblock {MacMahon}’s partition analysis: {I}. {The} lecture hall partition
  theorem.
\newblock In {\em Mathematical essays in honor of {Gian}-{Carlo} {Rota}}, pages
  1--22. Birkhäuser Boston, Boston, MA, 1998.

\bibitem{andrews_macmahons_2001}
G.~E. Andrews, P.~Paule, and A.~Riese.
\newblock Mac{M}ahon's partition analysis. {IX}. {$k$}-gon partitions.
\newblock {\em Bull. Austral. Math. Soc.}, 64(2):321--329, 2001.

\bibitem{andrews2001macmahon}
G.~E. Andrews, P.~Paule, and A.~Riese.
\newblock Mac{M}ahon's partition analysis: the {O}mega package.
\newblock {\em European J. Combin.}, 22(7):887--904, 2001.

\bibitem{bousquet-melou_lecture_1997}
M.~Bousquet-M\'{e}lou and K.~Eriksson.
\newblock Lecture hall partitions.
\newblock {\em Ramanujan J.}, 1(1):101--111, 1997.

\bibitem{dandrea_broken_2006}
C.~D'Andrea and E.~G\'{o}mez.
\newblock The broken spaghetti noodle.
\newblock {\em Amer. Math. Monthly}, 113(6):555--557, 2006.

\bibitem{flajolet2009analytic}
P.~Flajolet and R.~Sedgewick.
\newblock {\em Analytic combinatorics}.
\newblock Cambridge University Press, Cambridge, 2009.

\bibitem{kong2013random}
L.~Kong, L.~Lkhamsuren, A.~Turney, A.~Uppal, and A.~J. Hildebrand.
\newblock Random {Points}, {Broken} {Sticks}, and {Triangles} {Project}
  {Report}.
\newblock
  \url{https://faculty.math.illinois.edu/~hildebr/ugresearch/brokenstick-spring2013report.pdf},
  2013.

\bibitem{macmahon2001combinatory}
P.~A. MacMahon.
\newblock {\em Combinatory analysis}.
\newblock Cambridge University Press, Cambridge, 1915--1916.
\newblock Two volumes. Reprinted in one volume: Chelsea, New York, 1960.

\bibitem{oeis}
{OEIS Foundation Inc. (2021)}.
\newblock The {On-Line Encyclopedia of Integer Sequences,
  http://oeis.org/A008937.}

\bibitem{page_calculus_2015}
A.~Page, Y.~Semibratova, Y.~Xuan, E.~R. Zhang, M.~T. Phaovibul, and A.~J.
  Hildebrand.
\newblock Calculus, {Geometry}, and {Probability} in $n$ {Dimensions}: {The}
  {Broken} {Stick} {Problem} in {Higher} {Dimensions}, {IGL} {Project}
  {Report}.
\newblock
  \url{https://faculty.math.illinois.edu/~hildebr/ugresearch/Hildebrand-Calculus-Spring2015-report.pdf},
  2015.

\bibitem{petersen_broken_2020}
T.~K. Petersen and B.~E. Tenner.
\newblock Broken bricks and the pick-up sticks problem.
\newblock {\em Math. Mag.}, 93(3):175--185, 2020.

\bibitem{stanley_linear_1973}
R.~P. Stanley.
\newblock Linear homogeneous {D}iophantine equations and magic labelings of
  graphs.
\newblock {\em Duke Math. J.}, 40:607--632, 1973.

\bibitem{Verreault}
W.~Verreault.
\newblock On the probability of forming polygons from a broken stick.
\newblock {\em Statist. Probab. Lett.}, 180:Paper No. 109237, 2022.

\end{thebibliography}

\end{document}